\documentclass[12pt]{article}
\usepackage{amsfonts,amsmath,amsthm, amssymb}
\usepackage{latexsym, euscript, epic, eepic}
\newtheorem{lemma}{Lemma}[section]
\newtheorem{theorem}{Theorem}[section]
\newtheorem{definition}{Definition}[section]

\newtheorem{remark}{Remark}

\def \N{\mathcal{N}}

\begin{document}
\title{ On the Topological Pressure of Random Bundle Transformations in Sub-additive Case }
\author{ Yun Zhao, \,\,\,Yongluo Cao\\
\small\it Department of mathematics, Suzhou University\\
\small\it Suzhou 215006, Jiangsu, P.R.China,
  (ylcao@suda.edu.cn) }
\date{}
 \footnotetext{Partially supported by  NSFC(10571130), NCET, and  National Basic Research Program of China (973 Program)(2007CB814800).}
 \maketitle

\begin{center}
\begin{minipage}{120mm}
{\small {\bf Abstract.}  In this paper, we define  the topological
pressure for sub-additive potential via separated sets in  random
dynamical systems and  give a proof of the relativized
variational principle for the topological pressure.}
\end{minipage}
\end{center}

\vskip0.5cm

\vskip0.5cm

\section{Introduction.}
\noindent The setup consists of a probability space $(\Omega,
\mathcal{W},\mathbb{P})$, together with a $\mathbb{P}-$preserving
transformation $\vartheta$, of a compact metric space $X$ together
with the distance function $d$ and the Borel $\sigma-$algebra
${\mathcal{B}}_{X}$, and of a measurable set $ \mathcal E \subset
\Omega\times X$ and such that all the fibers ( sometimes called
$\omega$-sections ) $\mathcal E_{\omega}=\{ x\in X \mid
(\omega,x)\in \mathcal E\}$ are compact. We assume $\mathcal{W}$
is complete, countably generated, and separates points, and so
$(\Omega, \mathcal {W},\mathbb{P})$ is a Lebesgue space. A
continuous bundle RDS over $(\Omega,
\mathcal{W},\mathbb{P},\vartheta)$ is generated by mappings
$T_{\omega}: \mathcal E_{\omega}\rightarrow \mathcal
E_{\vartheta\omega}$ with iterates
$T_{\omega}^{n}={T_{\vartheta^{n-1}\omega } \cdots
T_{\vartheta\omega} T_{\omega}},n\geq 1$, so that the map
$(\omega,x)\mapsto T_{\omega}x$ is measurable and the map
$x\mapsto T_{\omega}x$ is continuous for $\mathbb{P}$- almost all
$\omega$, here and in what follows we think of $\mathcal
E_{\omega}$ being equipped with the trace topology, i.e. an open
set $A\subset \mathcal E_{\omega}$ is of the form $A=B\cap
\mathcal E_{\omega}$ with some open set $B\subset X$. The map
$$\Theta : \mathcal E\rightarrow \mathcal E, \,\,  \Theta(\omega,x)=(\vartheta\omega,T_{\omega}x)$$
  is called  the
skew product transformation.

Let $L_{\mathcal E}^{1}(\Omega,C(X))$ denote the collection of all
integrable random continuous functions on fibers, i.e. a
measurable $f: \mathcal E\rightarrow \mathbf{R}$ is a member of
$L_{\mathcal E}^{1}(\Omega,C(X))$ if $f(\omega, \cdot): \mathcal
E_{\omega} \rightarrow \mathbf{R}$ is continuous and $\| f
\|:=\int_{\Omega} |f(\omega)|_{\infty} \mathrm{d}
\mathbb{P}(\omega)<\infty$, where $|f(\omega)
|_{\infty}=\sup_{x\in \mathcal E_{\omega}} |f(\omega,x)|$. If we
identify $f$ and $g$ provided $\| f-g \|=0$, then $L_{\mathcal
E}^{1}(\Omega,C(X))$ becomes a Banach space with the norm $\|\cdot
\|$.

The family  $\mathcal{F}$ = $\{ f_{n} \}$$_{n=1}^{\infty}$ of
integrable random continuous functions on $\mathcal E$ is called
sub-additive if for $\mathbb{P}$-almost all  $\omega$,
 $$f_{n+m}(\omega,x)\leq
f_{n}(\omega,x)+f_{m}(\Theta^{n}(\omega,x)) \,\,\, \mbox{for all}
\,\, x \in \mathcal E_{\omega} .$$

In the special case in which the \( \vartheta \)-invariant measure
\( \mathbb{P}\) is a Dirac-\( \delta \) measure supported on a
single fixed point \( \{p\} \), it reduces to the case in which \(
T: X \to X \) is a standard deterministic dynamical system.

In deterministic dynamical systems $T: X \to X$, the topological
pressure for additive potential was first introduced by Ruelle
\cite{ruelle} for  expansive maps acting on compact metric spaces.
In the same paper he formulated a variational principle for the
topological pressure. Later Walters \cite{walter} generalized
these results to general continuous maps on compact metric spaces.
The theory about the topological pressure, variational principle
and equilibrium states plays a fundamental role in statistical
mechanics, ergodic theory and dynamical systems. The fact that the
topological pressure is a characteristic of dimension type  was
first noticed by Bowen \cite{bowen}. Since then, it has become the
main tool in studying dimension of invariant sets and measure for
dynamical systems and the dimension of cantor-like sets in
dimension theory.

In \cite{caofh}, authors generalize Ruelle and Walters's result to
sub-multiplicative potentials in general compact dynamical
systems. They define the sub-multiplicative topological pressure
 and give  a variational principle for the sub-multiplicative topological
 pressure. Then in \cite{cao}, author uses the variational principle for the sub-multiplicative topological
 pressure to give an upper bound estimate of Hausdorff dimension for
 nonconformal repeller,  which generalizes the results by Falconer in
 \cite{falconer}, Barreira in \cite{barreira1, barreira}, and Zhang in
 \cite{zhang}.

 We point out  that  Falconer  had some earlier contributions in the study of
thermodynamic formalism for sub-additive potentials. In
\cite{Fal88}, Falconer  considered the thermodynamic formalism for
sub-additive potentials on mixing repellers. He proved the
variational principle about the topological pressure  under some
Lipschitz conditions and bounded distortion assumptions on the
sub-additive potentials. More precisely, he assumed that there
exist constants $M, a,b>0$ such that
$$
\frac{1}{n}\left|\log f_n(x)\right|\leq M,\quad
\frac{1}{n}\left|\log f_n(x)-\log f_n(y)\right|\leq a|x-y|, \quad
\forall x,y\in X, n\in \N
$$
and $|\log f_n(x)-\log f_n(y)|\leq b$ whenever  $x,y$ belong to
the same $n$-cylinder of the mixing repeller $X$.

 In deterministic case, the thermodynamic formalism based on
 the statistical mechanics notions of pressure and equilibrium
 states plays an important role in the study of chaotic properties
 of random transformations. The first version of the relativized variational
 principle
 appeared  in \cite{led-walters} and later it was extended in
 \cite{bog} to random transformations for special potential
 function. In \cite{kifer}, Kifer extended the variational 
 principle of topological pressure 
 for general integrable random continuous function.

The aim of this paper is to introduce topological pressure of
random bundle transformations for sub-additive potentials, and
show a relativized variational principle. We can see it as an
extension of results in \cite{caofh}, \cite{kifer}. The paper is
organized in the following manner: in section 2 we introduce the
definitions. In section 3 we will provide some useful lemmas. In
section 4 we will state and prove the main theorem: the
relativized variational principle. In section 5 we will apply
topological pressure of random bundle transformations for
sub-additive potentials to obtain the Hausdorff dimension of
asymptotically conformal repeller.

\section{ Topological pressure and entropy of bundle RDS}

In this section, we give the definitions of entropy and the
topological pressure for sub-additive potential.

Denote by ${\mathcal{P}_{\mathbb{P}}} (\Omega\times X)$ the space
of probability measures on $\Omega\times X$ having the marginal
$\mathbb{P}$ on $\Omega$ and set ${\mathcal{P}_{\mathbb{P}}}
(\mathcal E)=\{ \mu\in {\mathcal{P}_{\mathbb{P}}} (\Omega\times
X):\mu(\mathcal E)=1 \}$. Any $\mu\in {\mathcal{P}_{\mathbb{P}}}
(\mathcal E)$ on $\mathcal E$ disintegrates ${\mathrm{d}}
\mu(\omega,x)= {\mathrm{d}}
\mu_{\omega}(x){\mathrm{d}}\mathbb{P}(\omega)$, where $\mu_\omega$
are regular conditional probabilities with respect to the
$\sigma-$algebra ${\mathcal{W}}_{\mathcal E}$ formed by all sets
$(A \times X)\cap \mathcal E$ with $A\in \mathcal{W}$. This means
that $\mu_{\omega}$ is a probability measure on $\mathcal
E_{\omega}$ for $\mathbb{P}-a.a.\omega$ and for any measurable set
$R\subset \mathcal E$, $\mathbb{P}-$a.s.
$\mu_{\omega}(R_\omega)=\mu(R|{\mathcal{W}}_{\mathcal E})$, where
$R_\omega=\{ x: (\omega,x)\in R\}$, and so $\mu(R)=\int
\mu_{\omega}(R_\omega)\mathrm{d}\mathbb{P}(\omega)$. Now let
${\mathcal{R}}=\{ R_{i} \}$ be a finite or countable partition of
$\mathcal E$ into measurable sets. Then ${\mathcal{R}}(\omega)=\{
R_{i}(\omega)\}$, $R_i(\omega)=\{ x\in \mathcal E_\omega :
(\omega,x)\in R_i\}$ is a partition of $\mathcal E_\omega$. The
conditional entropy of $\mathcal{R}$ given the $\sigma-$algebra
${\mathcal{W}}_\mathcal E$ is defined by
\begin{eqnarray} H_{\mu}({\mathcal{R}}\mid
{\mathcal{W}}_\mathcal E)&=&- \int \sum\limits_i \mu(R_i\mid
{\mathcal{W}}_\mathcal E) \log \mu(R_i\mid {\mathcal{W}}_\mathcal
E) \mathrm{d}
\mathbb{P}\\
&=&\int H_{\mu_{\omega}}(\mathcal{R}(\omega))\mathrm{d}\mathbb{P}
\end{eqnarray}
where $H_{\mu_{\omega}}(\mathcal{A})$ denotes the usual entropy of
a partition $\mathcal{A}$. Let
${\mathcal{M}}_{\mathbb{P}}^{1}(\mathcal E,T)$ denote the set of
$\Theta-$invariant measures
$\mu\in{\mathcal{P}}_{\mathbb{P}}(\mathcal E)$. The entropy
$h_{\mu}^{(r)}(T)$ of the RDS $T$ with respect to $\mu$ is defined
by the formula \begin{eqnarray}
h_{\mu}^{(r)}(T)=\sup\limits_{\mathcal{Q}}
h_{\mu}^{(r)}(T,\mathcal{Q}) \end{eqnarray}
 where
$h_{\mu}^{(r)}(T,\mathcal{Q})$=$\lim\limits_{n\rightarrow \infty}
\frac{1}{n} H_{\mu}(\bigvee\limits_{i=0}^{n-1}(\Theta^{i})^{-1}
{\mathcal{Q}}\mid {\mathcal{W}}_\mathcal E)$ and the supremum is
taken over all finite or countable measurable partitions
${\mathcal{Q}}=\{ Q_i \}$ of $\mathcal E$ with
$H_{\mu}({\mathcal{Q}}\mid {\mathcal{W}}_\mathcal E)<\infty$.

Observe that if ${\mathcal{Q}}=\{ Q_i \}$ is a partition of
$\mathcal E$, then
 $\mathcal{R}=$$\bigvee\limits_{i=0}^{n-1}(\Theta^{i})^{-1}\mathcal{Q}$
 is a partition of $\mathcal E$ consisting of sets $\{ R_j \}$ such that
 the corresponding partition ${\mathcal{R}}(\omega)=\{ R_{j}(\omega)
 \}$, $R_{j}(\omega)=\{ x: (\omega,x)\in R_{j} \}$ of $\mathcal E_\omega$ has
 the form ${\mathcal{R}}(\omega)=\bigvee\limits_{i=0}^{n-1}(T_{\omega}^{i})^{-1}
 {\mathcal{Q}}(\vartheta^{i}\omega)$, where ${\mathcal{Q}}(\omega)=\{ Q_{i}(\omega)
 \}$, $Q_i(\omega)=\{ x\in \mathcal E_{\omega}: (\omega,x)\in Q_i \}$
 is a partition of  $\mathcal E_\omega$. So
 \begin{eqnarray} h_{\mu}^{(r)}(T,{\mathcal{Q}})=\lim\limits_{n\rightarrow
 \infty} \frac{1}{n} \int H_{\mu_{\omega}} (\bigvee\limits_{i=0}^{n-1}(T_{\omega}^{i})^{-1}{\mathcal{Q}}(\vartheta^{i}\omega))
 \mathrm{d} \mathbb{P}(\omega)
 \end{eqnarray} In \cite{bog} and \cite{ki}, the authors say that the
 resulting entropy remains the same if we take the supremum in (2.3)
 only over partitions $\mathcal{Q}$  of $\mathcal E$ into sets $Q_i$ of the
 form $Q_i=(\Omega\times A_i)\cap \mathcal E$, where ${\mathcal{A}}=
 \{ A_i \}$ is a partition of $
 X$ into measurable sets, so that $Q_i(\omega)=A_i\cap \mathcal E_\omega$. If $\vartheta$ is invertible,
 then $\mu\in {\mathcal{P}}_{\mathbb{P}}(\mathcal E)$ is $\Theta-$invariant if and only if
 the disintegrations $\mu_{\omega}$ of $\mu$ satisfy
 $T_{\omega}\mu_{\omega}=\mu_{\vartheta\omega}$ $\mathbb{P}-a.s.$  In this case,
 if, in addition, $\mathbb{P}$ is ergodic, then the formula(2.4)
 remains true $\mathbb{P}-$a.s. without integrating against
 $\mathbb{P}$.

For each $n\in \mathbf{N}$ and a positive random variable 
$\epsilon=\epsilon(\omega)$, we define a family of metrics
$d_{\epsilon,n}^{\omega}$ on $\mathcal E_{\omega}$ by the 
formula
\[ d_{\epsilon,n}^{\omega}(x,y)=
\max\limits_{0\leq k< n} (d(T_{\omega}^{k}y,T_{\omega}^{k}x)
\times (\epsilon(\vartheta^k \omega ))^{-1}),\qquad x,y\in \mathcal E_{\omega}
\]
where $T_{\omega}^{0}$ is the identity map. In \cite{kifer}, the
author proves that $d_{\epsilon,n}^{\omega}(x,y)$ depends
measurably on $(\omega,x,y)\in \mathcal E^{(2)}:=\{ (\omega,x,y):
x,y\in \mathcal E_{\omega} \}$. Denote by
$B_{\omega}(n,x,\epsilon)$ the closed ball in $\mathcal
E_{\omega}$ centered at $x$ of radius 1 with respect to the metric
$d_{\epsilon,n}^{\omega}$. For $d_{\epsilon,1}^{\omega}$ and
$B_{\omega}(1,x,\epsilon)$,  we will write simply
$d_{\epsilon}^{\omega}$ and $B_{\omega}(x,\epsilon)$ respectively.
We say that $x,y\in \mathcal E_{\omega}$ are $(\omega,\epsilon,
n)-$close if $d_{\epsilon,n}^{\omega}(x,y)\leq 1$.

\begin{definition} \label{fenli} \rm
\it A set $F\subset \mathcal E_{\omega}$ is said to be
$(\omega,\epsilon,n)-$separated for $T$, if $x,y\in F,x\neq y$
implies $d_{\epsilon,n}^{\omega}(x,y)> 1$.
\end{definition}

It is easy to see that if $F$ is maximal
$(\omega,\epsilon,n)-$separated, i.e. for every $x\in \mathcal
E_{\omega}$ with $x\not\in F$ the set $F\cup \{x\}$ is not
$(\omega,\epsilon,n)-$separated anymore, then $\mathcal
E_{\omega}=\bigcup_{x\in F} B_{\omega}(n,x,\epsilon)$. Due to the
compactness of $\mathcal E_{\omega}$, there exists a maximal $
(\omega,\epsilon,n)-$separated set $F$ with finite elements.

Let ${\mathcal{F}}=\{ f_n \}$ be a sub-additive function sequence
with $f_n \in L_{\mathcal E}^{1}(\Omega,C(X))$ for each $n$.  As
usual for any $n\in \mathbf{N}$ and a positive random variable $\epsilon$,   we define
\[ \pi_T(\mathcal{F}) (\omega,\epsilon,n )= \sup \{ \sum\limits_{x\in
F} e^{f_{n}(\omega,x)}\mid F \ is\  an \ (\omega,\epsilon,n) \
-separated\  subset \ of \ \mathcal E_{\omega}\}
\] and
\begin{eqnarray*}
&&{ \pi_T(\mathcal{F})}(\epsilon)=\limsup\limits_{n\rightarrow
\infty} \frac{1}{n} \int \log{\pi_T(\mathcal{F})}(\omega,\epsilon,n) \mathrm{d} \mathbb{P}(\omega),\\
&&{ \pi_T(\mathcal{F})}=\lim\limits_{\epsilon
\downarrow0}{\pi_T(\mathcal{F})}(\epsilon).
\end{eqnarray*}
By lemma 3.1 in section 3, we know that the definition of
$\pi_T(\mathcal{F})(\epsilon)$ is reasonable. The last limit
exists since ${\pi_T(\mathcal{F})}(\epsilon)$ is monotone in
$\epsilon$. In fact, $\lim_{\epsilon\rightarrow 0}$ as above
equals to $\sup_{\epsilon
>0}$.

\begin{remark} In \cite{kifer}, the author defined additive
topological pressure for a random positive variable $\epsilon$,
but the limit should be taken over some directed sets. We can find 
  detailed description  of difference
between random and nonrandom case of $\epsilon$ in
\cite{bog2}.
\end{remark}

\section{Some Lemmas}

\noindent In this section, we will give some lemmas which will be
used in our proof of the main theorem in the next section.

Let ${\mathcal{F}}, T$ and $\pi_T(\mathcal{F})$ be defined as in
section 2, and $(X,d)$ be a compact metric space. Notice that if
$\mu\in {\mathcal{M}}_{\mathbb{P}}^{1}(\mathcal E,T)$, then we let
${\mathcal{F}}_{*}(\mu)$ denote the following limit
${\mathcal{F}}_{*}(\mu)=\lim\limits_{n\rightarrow\infty}
\frac{1}{n} \int_{\mathcal E} f_{n} {\mathrm{d}}\mu$. The
existence of the limit follows from a sub-additive argument. We
begin with the following lemmas, and we point out that the proof of the
first two lemmas can be easily obtained by following the proof
 in \cite{kifer}. We cite here just for complete.

\begin{lemma} \label{reason} \rm
\it  For any $n\in \mathbf{N}$ and a positive random variable
$\epsilon=\epsilon(\omega)$ the function
$\pi_T({\mathcal{F}})(\omega,\epsilon,n)$ is measurable in
$\omega$, and for each $\delta>0$ there exits a family of maximal
$(\omega,\epsilon,n)$ separated sets $G_\omega\subset \mathcal
E_\omega$ satisfying
\[ \sum\limits_{x\in G_\omega} e^{f_{n}(\omega,x)} \geq (1-\delta) \pi_T({\mathcal{F}})(\omega,\epsilon,n)
\]
and depending measurably on $\omega$ in the sense that $G=\{
(\omega,x): x\in G_\omega \}\in \mathcal{W} \times \mathcal B_X$,
which also means that the mapping $\omega\mapsto G_\omega$ is
measurable with respect to the Borel $\sigma-$algebra induced by
the Hausdorff topology on the space $\mathcal K(X)$ of compact subsets
of $X$. In particular, the supremum in the definition of
$\pi_T({\mathcal{F}})(\omega,\epsilon,n)$ can be taken only over
measurable in $\omega$ families of $(\omega,\epsilon,n)$
-separated sets.
\end{lemma}

\begin{lemma} \label{ceducpt} \rm
\it For $\mu,\mu_{n}\in {\mathcal{P}}_{\mathbb{P}}(\mathcal E)$,
$n=1,2,\dots$ , write $\mu_n \Rightarrow \mu $ if $\int f
\mathrm{d} \mu_n \rightarrow \int f \mathrm{d}\mu$ as
$n\rightarrow\infty$ for any $f\in L_{\mathcal
E}^{1}(\Omega,C(X))$ that introduces a weak* topology in
${\mathcal{P}}_{\mathbb{P}}(\mathcal E)$. Then

(i)the space ${\mathcal{P}}_{\mathbb{P}}(\mathcal E)$ is compact
in this weak* topology;

 (ii) for any sequence $\upsilon_k\in
 {\mathcal{P}}_{\mathbb{P}}(\mathcal E)$, $k=0,1,2,\dots$, the set of limit
points in the above weak* topology of the sequence
\[ \mu_n=\frac{1}{n}\sum\limits_{k=0}^{n-1}\Theta^{k}\upsilon_n\qquad
as\ n\rightarrow\infty
\]
is not empty and is contained in
${\mathcal{M}}_{\mathbb{P}}^{1}(\mathcal E,T)$;

(iii) let $\mu,\mu_n\in {\mathcal{P}}_{\mathbb{P}}(\mathcal E)$ ,
$n=1,2,\dots$, and $\mu_n\Rightarrow \mu$ as $n\rightarrow\infty$;
let ${\mathcal{P}}=\{ P_1,P_2,\dots,P_k \}$ be a finite partition
of $X$ satisfying $\int
\mu_{\omega}(\partial{\mathcal{P}}_{\omega})
\mathrm{d}\mathbb{P}(\omega)=0$, where
$\partial{\mathcal{P}}_{\omega}=\bigcup_{i=1}^{k}
\partial(P_i\cap \mathcal E_\omega)$ is the boundary of ${\mathcal{P}}_{\omega}=\{ P_1\cap \mathcal E_\omega,\dots,P_k\cap \mathcal E_\omega
\}$; denote by $\mathcal{R}$ the partition of $\Omega\times X$
into sets $\Omega\times P_i$; then
\[ \limsup\limits_{n\rightarrow\infty} H_{\mu_{n}}({\mathcal{R}}\mid
{\mathcal{W}}_\mathcal E)\leq H_{\mu}({\mathcal{R}}\mid
{\mathcal{W}}_\mathcal E)
\]
\end{lemma}

\begin{lemma} \label{bigsll} \rm
\it For any $k\in \mathbf{N}$, we have
\[ \pi_{T^{k}}({\mathcal{F}}^{(k)}) \leq k \pi_{T}(\mathcal{F})
\]
where $(T^{k})_{\omega}:=T_{\vartheta^{k-1}\omega}\circ\dots\circ
T_{\vartheta \omega}\circ T_{\omega} $ and
${\mathcal{F}}^{(k)}:=\{ f_{kn} \}_{n=1}^{\infty}$ .
\end{lemma}
\begin{proof} Fix $k\in \mathbf{N}$. Note that if $F$ is an
$(\omega,\epsilon,n)$ separated set for $T^{k}$ of $\mathcal
E_\omega$, then $F$ is an $(\omega,\epsilon,kn)$ separated set for
$T$ of $\mathcal E_\omega$. It follows that
\begin{eqnarray*}
\pi_{T}({\mathcal{F}})(\omega, \epsilon, kn)&=&\sup \{
\sum\limits_{x\in F} e^{f_{kn}(\omega,x)} : F\ is\ an\
(\omega,\epsilon,kn)\ separated\ subset\ of\ \mathcal E_\omega\ for\ T \}\\
&\geq& \sup \{ \sum\limits_{x\in F} e^{f_{kn}(\omega,x)} : F\ is\
an\ (\omega,\epsilon,n)\ separated\ subset\ of\ \mathcal E_\omega\ for\ T^{k} \}\\
&=&\pi_{T^{k}}({\mathcal{F}}^{k})(\omega,\epsilon,n).
\end{eqnarray*}
It implies that $\pi_{T^{k}}({\mathcal{F}}^{k})\leq k
\pi_{T}(\mathcal{F})$. \end{proof}

\begin{lemma} \label{bigsll2} \rm
 \it For any positive integer $k$ and $\mu\in
{\mathcal{P}}_{\mathbb{P}}(\mathcal E)$, we have
\[ \int_{\mathcal E} k f_{n}(\omega,x) {\mathrm{d}} \mu \leq 4k^{2}C +
\int_{\mathcal E} \sum\limits_{i=0}^{n-1}
f_{k}(\Theta^i(\omega,x)) \mathrm{d} \mu
\]
where $C=\|f_1 \|$.
\end{lemma}
\begin{proof}  For a  fixed $k$, it has $n= ks + l, \, 0\le l <
k$. For $j=0,1,\dots,k-1$, the subadditivity of $f_n(\omega,x)$ implies that 
\[ f_{n}(\omega,x) \leq  f_{j}(\omega,x)+ f_k(\Theta^j(\omega,x)) + \cdots +
f_k(\Theta^{k(s-2)} \Theta^j(\omega,x)) +
f_{k+l-j}(\Theta^{k(s-1)} \Theta^j(\omega,x)).
\]
Hence
\begin{eqnarray*}
\int_{\mathcal E}  f_{n}(\omega,x) {\mathrm{d}}\mu &\leq&
\int_{\mathcal E} f_{j}(\omega,x) {\mathrm{d}}\mu + \int_{\mathcal
E}\sum\limits_{i=0}^{s-2} f_k(\Theta^{ki}
\Theta^j(\omega,x)){\mathrm{d}}\mu +
\int_{\mathcal E}f_{k+l-j}(\Theta^{k(s-1)} \Theta^j(\omega,x)){\mathrm{d}}\mu \\
&\le& \|f_j\| + \int_{\mathcal E}\sum\limits_{i=0}^{s-2}
f_k(\Theta^{ki} \Theta^j(\omega,x)){\mathrm{d}}\mu +
\|f_{k+l-j}\|\\&\le& 2k \|f_1\| + \int_{\mathcal
E}\sum\limits_{i=0}^{s-2} f_k(\Theta^{ki}
\Theta^j(\omega,x)){\mathrm{d}}\mu.
\end{eqnarray*}
Summing $j$ from $0$ to $k-1$, we get
 \begin{eqnarray*}
  \int_{\mathcal E} kf_{n}(\omega,x) {\mathrm{d}}\mu &\leq& 2k^{2}C +
  \int_{\mathcal E}
  \sum\limits_{i=0}^{k(s-1)-1}f_k(\Theta^i(\omega,x)) {\mathrm{d}}\mu  \\&=& 2k^{2}C + \int_{\mathcal E}
  \sum\limits_{i=0}^{n-1}f_k(\Theta^i(\omega,x)) {\mathrm{d}}\mu -
  \int_{\mathcal E}
  \sum\limits_{i=k(s-1)}^{n-1}f_k(\Theta^i(\omega,x)) {\mathrm{d}}\mu  \\ &\le& 4k^2 C + \int_{\mathcal E}
  \sum\limits_{i=0}^{n-1}f_k(\Theta^i(\omega,x)) {\mathrm{d}}\mu.
\end{eqnarray*}
This finishes the proof of the lemma.
\end{proof}

\begin{lemma} \label{converge} \rm
\it  Let $m^{(n)}$ be a sequence in $
{\mathcal{P}}_{\mathbb{P}}(\mathcal
 E)$. The new sequence
$\{\mu^{(n)} \} _{n=1}^{\infty}$ is defined as
$\mu^{(n)}=\frac{1}{n} \sum\limits_{i=0}^{n-1} \Theta^{i}
m^{(n)}$. Assume $\mu^{(n_i)}$ converges to $\mu$ in $\mathcal
P_{\mathbb{P}}(\mathcal E)$ for some subsequence $\{n_i\}$. Then $\mu\in
{\mathcal{M}}_{\mathbb{P}}^{1}(\mathcal E,T)$, and moreover
\[ \limsup\limits_{i\rightarrow \infty} \frac{1}{n_i} \int_{\mathcal E} f_{n_i}(\omega,x)
{\mathrm{d}}m^{(n_i)}(\omega,x) \leq {\mathcal{F}}_{*}(\mu)
\]
where ${\mathcal{F}}_{*}(\mu)=\inf\limits_{n} \{ \frac{1}{n}
\int_{\mathcal E} f_n(\omega,x) {\mathrm{d}}\mu \}$.
\end{lemma}
\begin{proof} The first statement $\mu\in
{\mathcal{M}}_{\mathbb{P}}^{1}(\mathcal E,T)$ is contained in
lemma 3.2. To show the desired inequality, we fix $k\in
\mathbf{N}$. By lemma 3.4,  we have
\begin{eqnarray*}
\frac{1}{n} \int_{\mathcal E} f_{n}(\omega,x) {\mathrm{d}}
m^{(n)}&=&\frac{1}{kn} \int_{\mathcal E}
kf_{n}(\omega,x) {\mathrm{d}} m^{(n)}\\
&\leq&\frac{1}{kn} (4k^{2}C+ \int_{\mathcal E}
\sum\limits_{j=0}^{n-1}
f_{k}(\Theta^j(\omega,x)){\mathrm{d}}m^{(n)})\\
&=&\frac{4kC}{n} +  \int_{\mathcal E} \frac1k f_{k}(\omega,x) {\mathrm{d}}
\mu^{(n)}.
\end{eqnarray*}
In particularly, we have
\begin{eqnarray*}
\frac{1}{n_i} \int_{\mathcal E} f_{n_i}(\omega,x) {\mathrm{d}} m^{(n_i)}
\leq  \int_{\mathcal E} \frac1k f_{k}(\omega,x) {\mathrm{d}} \mu^{(n_i)} +
\frac{4kC}{n_i}.
\end{eqnarray*}
Since $\lim\limits_{i\rightarrow \infty}\mu^{(n_i)}= \mu$, we have
 \[ \limsup\limits_{i\rightarrow \infty} \frac{1}{n_i}
 \int_{\mathcal E}
 f_{n_i}(\omega,x) {\mathrm{d}} m^{(n_i)} \leq  \int_{\mathcal E}\frac{1}{k} f_k(\omega,x)
 {\mathrm{d}}\mu.
 \]
Letting $k$ approach infinity and applying the sub-additive
ergodic theorem, we have the desired result.
\end{proof}

\section{The statement of main  theorem and its proof}

\noindent  For random dynamical systems, the topological pressure
for sub-additive potential also has variational principle which
can be considered as a generalization of   variational principle
of topological pressure for  sub-additive potential in
deterministic dynamical systems in \cite{caofh}. Next we give a
statement of main theorem and its proof.

\begin{theorem} \label{mthm} \rm
\it Let $\Theta $ be a continuous bundle random dynamical systems
on $\mathcal E$, and $\mathcal{F}$ a sequence of sub-additive
random continuous functions in $L^1_{\mathcal E}(\Omega,C(X))$.
Then
 \[
  \pi_{T}(\mathcal{F})=\left\{
     \begin{array}{cc}
     -\infty, &if\ {\mathcal{F}}_{*}(\mu)=-\infty\ for\ all\ \mu \in
     {\mathcal{M}}_{\mathbb{P}}^{1}(\mathcal E,T)\\
     \sup \{ h_{\mu}^{(r)}(T) +{\mathcal{F}}_{*}(\mu) : \mu \in
     {\mathcal{M}}_{\mathbb{P}}^{1}(\mathcal E,T) \}, &otherwise.
     \end{array}
     \right.
\]
\end{theorem}
\begin{proof} For clarity, we divide the proof into three small steps:\\
Step 1: $\pi_{T}({\mathcal{F}}) \geq h_{\mu}^{(r)}(T) +
{\mathcal{F}}_{*}(\mu)$ , $\forall  \mu\in
{\mathcal{M}}_{\mathbb{P}}^{1}(\mathcal E,T)\ with\
{\mathcal{F}}_{*}(\mu)\neq -\infty$.

 Let  $\mu \in
{\mathcal{M}}_{\mathbb{P}}^{1}(\mathcal E,T)$ satisfying
${\mathcal{F}}_{*}(\mu)\neq -\infty$ and ${\mathcal{A}}=\{
A_1,\dots,A_k \}$ be a finite  partition of $X$. Let $\alpha>0$ be
given. Choose $\epsilon >0$ so that $\epsilon k \log k <\alpha$.
Denote by ${\mathcal{A}}(\omega)=\{ A_1(\omega),\dots,A_k(\omega)
\},\,\, A_i(\omega)=A_i\cap \mathcal E_\omega, i=1,\dots,k$, the
corresponding partition of $\mathcal E_\omega$. By the regularity
of $\mu$, we can find compact sets $B_i\subset A_i$, $1\leq i\leq
k$, such that
\[ \mu (A_i\setminus B_i) =\int \mu_{\omega}(A_i(\omega)\setminus
B_i(\omega)) {\mathrm{d}}\mathbb{P}(\omega) < \epsilon,
\]
where $B_i(\omega)=B_i \cap  \mathcal E_\omega$. Then let
$B_0(\omega) =\mathcal E_\omega \setminus \bigcup_{i=1}^{k}
B_i(\omega)$.  It follows that $$\int \mu_{\omega}(B_0(\omega))
{\mathrm{d}}{\mathbb{P}}(\omega)<k\epsilon.$$  Therefore (see
\cite{ki} p.79) the partition ${\mathcal{B}}(\omega)=\{
B_0(\omega),\dots,B_k(\omega) \}$ satisfies the inequality
\[ H_{\mu_{\omega}} ({\mathcal{A}}(\omega)\mid
{\mathcal{B}}(\omega)) \leq \mu_{\omega}(B_0(\omega)) \log k.
\]
Hence
\[
\int H_{\mu_{\omega}} ({\mathcal{A}}(\omega)\mid
{\mathcal{B}}(\omega)) {\mathrm{d}}{\mathbb{P}} \leq \int
\mu_{\omega}(B_0(\omega)) \log k {\mathrm{d}}{\mathbb{P}} \leq
k\epsilon \log k < \alpha.
\]

Take any $\omega$ such that $\mathcal E_\omega$ makes sense, and set
$b=\min\limits_{1\leq i\neq j\leq k} d(B_i,B_j)>0$. Pick $\delta
>0$ so that $\delta < b/2$. Let $n\in \mathbf{N}$. For each $C\in
{\mathcal{B}}_{n}(\omega):=\bigvee_{j=0}^{n-1} (T_{w}^{j})^{-1}
{\mathcal{B}}(\vartheta^{j}(\omega))$, choose some $x(C)\in$
Closure(C) such that $f_n(\omega, x(C))=\sup \{ f_n(\omega,x):
x\in C \}$, and we claim that for each $C\in
{\mathcal{B}}_n(\omega)$, there are at most $2^n$ many different
$\tilde{C}$'s in ${\mathcal{B}}_n(\omega)$ such that
\[ d_{\delta,n}^{\omega}(x(C),x(\tilde{C})):=\max\limits_{0\leq j\leq n-1}
 d(T_{\omega}^{j}x(C),T_{\omega}^{j}x(\tilde{C}))\delta^{-1} \leq 1 .\]
To see this claim, for each $C\in {\mathcal{B}}_n(\omega)$ we pick
up the unique index $(i_0(C),i_1(C),...,i_{n-1}(C))$ $\in \{
0,1,...,k \}^{n}$ such that
\[ C=B_{i_{0}(C)}(\omega)\cap (T_{\omega}^{1})^{-1}B_{i_{1}(C)}
(\vartheta\omega)\cap (T_{\omega}^{2})^{-1}B_{i_{2}(C)}(\vartheta^{2}\omega)
\cap ...\cap
(T_{\omega}^{n-1})^{-1}B_{i_{n-1}(C)}(\vartheta^{n-1}\omega).
\]
Now fix a $C\in {\mathcal{B}}_n(\omega)$ and let $\mathcal{Y}$
denote the collection of all $\tilde{C}\in
{\mathcal{B}}_n(\omega)$ with
$$d_{\delta,n}^{\omega}(x(C),x(\tilde{C})) \,\, \leq 1.$$  Then we have
\begin{eqnarray}\label{number} \#\{ i_l(\tilde{C}): \tilde{C}\in {\mathcal{Y}}
\}\leq 2,\quad l=0,1,...,n-1.
\end{eqnarray}
To see this inequality, we assume on the contrary that there are
three elements $\tilde{C_1}$,$\tilde{C_2}$,$\tilde{C_3}\in
\mathcal{Y}$ corresponding to the distinct values
$i_l(\tilde{C_1})$,$i_l(\tilde{C_2})$,$i_l(\tilde{C_3})$ for some
$0\leq l\leq n-1$ respectively. Then without  loss of generality,
we may assume $i_l(\tilde{C_1})\neq 0$ and $i_l(\tilde{C_2})\neq
0$. This implies
\begin{eqnarray*}
 d_{\delta,n}^{\omega}(x(\tilde{C_1}),x(\tilde{C_2})) &\geq&
 d(T_\omega^{l}x(\tilde{C_1}),T_\omega^{l}x(\tilde{C_2}))\delta^{-1}\geq
 d(B_{i_{l}(\tilde{C_{1}})}(\vartheta^{l}\omega),B_{i_{l}(\tilde{C_{2}})}(\vartheta^{l}\omega))\delta^{-1}\\
 &\geq& d(B_{i_{l}(\tilde{C_{1}})},B_{i_{l}(\tilde{C_{2}})})\delta^{-1}\geq b\delta^{-1}
 > 2\\ &\geq&
 d_{\delta,n}^{\omega}(x(\tilde{C_1}),x(\tilde{C}))+d_{\delta,n}^{\omega}(x(\tilde{C}),x(\tilde{C_2})).
 \end{eqnarray*}
 Which leads to a contradiction, thus (\ref{number}) is true, from which the claim follows. The third
 inequality follows from the fact that
$B_{i_{l}(\tilde{C_{j}})}(\vartheta^{l}\omega)=B_{i_{l}(\tilde{C_{j}})}\cap
\mathcal E_{\vartheta^{l}\omega}\subseteq
B_{i_{l}(\tilde{C_{j}})}(j=1,2)$.

 In the following we will construct an $(\omega,\delta,n)-$separated 
 set $G$ of $\mathcal E_\omega$  for $T$ such that
\begin{eqnarray}\label{inequality}
2^n \sum\limits_{y\in G}e^{f_n(\omega,y)} \geq \sum\limits_{C\in
{\mathcal{B}}_n(\omega)} e^{f_n(\omega,x(C))}.
\end{eqnarray}

(I) Take an element $C_1\in {\mathcal{B}}_n(\omega)$ such that
$f_n(\omega,x(C_1))=\max_{C\in {\mathcal{B}}_n(\omega)}
f_n(\omega,x(C))$. Let ${\mathcal{Y}}_1$ denote the collection of
all $\widetilde{C} \in {\mathcal{B}}_n(\omega)$ with
$d_{\delta,n}^{\omega}(x(\widetilde{C}),x(C_1))\leq1$ . Then the
cardinality of ${\mathcal{Y}}_1$ does not exceed $2^n$.

(II) If the collection ${\mathcal{B}}_n(\omega)\setminus
{\mathcal{Y}}_1$ is not empty, we choose an element $C_2\in
{\mathcal{B}}_n(\omega)\setminus {\mathcal{Y}}_1$ such that
$f_n(\omega,x(C_2))=\max_{C\in {\mathcal{B}}_n(\omega)\setminus
{\mathcal{Y}}_1} f_n(\omega,x(C))$. Let ${\mathcal{Y}}_2$ denote
the collection of $\widetilde{C}\in
{\mathcal{B}}_n(\omega)\setminus {\mathcal{Y}}_1$  with
$d_{\delta,n}^{\omega}(x(\widetilde{C}),x(C_2))\leq1$. We continue
this process. More precisely in step $m$, we choose an element
$C_m\in {\mathcal{B}}_n(\omega)\setminus
\bigcup\limits_{j=1}^{m-1} {\mathcal{Y}}_j$ such that
\[ f_n(w,x(C_m))=\max\limits_{C\in {\mathcal{B}}_n(\omega)\setminus
\bigcup\limits_{j=1}^{m-1} {\mathcal{Y}}_j} f_n(\omega,x(C)).
\]
Let ${\mathcal{Y}}_m$ denote the set of all ${\widetilde{C}}\in
{\mathcal{B}}_n(\omega)\setminus \bigcup\limits_{j=1}^{m-1}
{\mathcal{Y}}_j$ with
$d_{\delta,n}^{\omega}(x(\widetilde{C}),x(C_m))\leq1$. Since the
partition ${\mathcal{B}}_n(\omega)$ is finite, the above process
will stop at some step $l$. Denote $G=\{ x(C_j) : 1\leq j\leq l
\}$. Then $G$ is a $(\omega,\delta,n)-$separated set and
\[ \sum\limits_{y\in G} e^{f_n(\omega,y)} = \sum\limits_{j=1}^{l}
e^{f_n(\omega,x(C_j))} \geq  \sum\limits_{j=1}^{l} 2^{-n}
\sum\limits_{C\in {\mathcal{Y}}_{j}} e^{f_n(\omega,x(C))} =
2^{-n} \sum\limits_{C\in {\mathcal{B}}_{n}(\omega)}
e^{f_n(\omega,x(C))},
\]
from which (\ref{inequality}) follows.

Let $\mu\in {\mathcal{M}}_{\mathbb{P}}^{1}(\mathcal E,T)$. Then
\begin{eqnarray*}
 &&H_{\mu_{\omega}}({\mathcal{B}}_{n}(\omega)) + \int_{\mathcal E_\omega}
 f_n(\omega,x) {\mathrm{d}} \mu_{\omega}(x) \\
 \leq &&\sum\limits_{C\in {\mathcal{B}}_{n}(\omega)}
 \mu_{\omega}(C) (f_n(\omega,x(C)) -\log \mu_{\omega}(C))\\
\leq &&\log \sum\limits_{C\in {\mathcal{B}}_{n}(\omega)}
 e^{f_n(\omega,x(C))}\\
\leq &&\log 2^{n} \sum\limits_{y\in G} e^{f_n(\omega,y)}\\
 = && n\log 2 + \log  \sum\limits_{y\in G} e^{f_n(\omega,y)},
 \end{eqnarray*}
 the second inequality  follows from the standard inequality:
 $\Sigma p_i(a_i-\log p_i)\leq \log \Sigma e^{a_i}$ for any
 probability vector $(p_1,p_2,...,p_m)$, and the equality holds if
 and only if $p_i=e^{a_i}/\Sigma e^{a_j}$.
 Integrating against ${\mathbb{P}}$ on both sides of the above
 inequality, and dividing by $n$, we have
 \begin{eqnarray*}
 \frac{1}{n} \int H_{\mu_{\omega}}({\mathcal{B}}_{n}(\omega))
 {\mathrm{d}} {\mathbb{P}}(\omega) + \frac{1}{n} \int f_n(\omega,x)
 {\mathrm{d}} \mu(\omega,x) \leq \log 2 + \frac{1}{n} \int \log  \sum\limits_{y\in G}
 e^{f_n(\omega,y)} {\mathrm{d}} {\mathbb{P}}(\omega).
 \end{eqnarray*}
Letting $n\rightarrow \infty$, we obtain
\[ h_{\mu}^{(r)} (T,\Omega\times {\mathcal{B}}) +
{\mathcal{F}}_{*}(\mu) \leq \log 2 +
\pi_{T}({\mathcal{F}})(\delta).\]
 Using corollary 3.2 in \cite{bog}, we have
 \begin{eqnarray*}
  h_{\mu}^{(r)} (T,\Omega\times {\mathcal{A}}) +
  {\mathcal{F}}_{*}(\mu)&\leq& h_{\mu}^{(r)} (T,\Omega\times
  {\mathcal{B}}) + \int H_{\mu_{\omega}} ({\mathcal{A}}(\omega)\mid
{\mathcal{B}}(\omega)) {\mathrm{d}} {\mathbb{P}}
+{\mathcal{F}}_{*}(\mu)\\
 &\leq& \log 2 + \alpha +
\pi_{T}({\mathcal{F}})(\delta).
\end{eqnarray*}
Since this is true for all ${\mathcal{A}}$ , $\alpha$  and
$\delta$, we know
\[ h_{\mu}^{(r)} (T) + {\mathcal{F}}_{*}(\mu) \leq \log 2 +
\pi_{T}({\mathcal{F}}). \]
 Applying the above argument to $T^n$ and ${\mathcal{F}}^{(n)}$,
  since $h_{\mu}^{(r)}(T^n) =n h_{\mu}^{(r)}(T)$(see theorem3.6 in
 \cite{bog}), and using lemma 3.3, we obtain

 \begin{eqnarray*}
 n(h_{\mu}^{(r)} (T) + {\mathcal{F}}_{*}(\mu)) &\leq& \log 2 +
 \pi_{T^{n}}({\mathcal{F}}^{(n)})\\
 &\leq& \log 2 + n \pi_{T}({\mathcal{F}}).
 \end{eqnarray*}
 Since $n$ is arbitrary, we have $h_{\mu}^{(r)} (T) + {\mathcal{F}}_{*}(\mu) \leq
 \pi_{T}({\mathcal{F}})$.

 Step 2: If $\pi_{T}({\mathcal{F}})\neq -\infty$, then for any small
 enough $\epsilon >0$, there exists a $\mu\in
 {\mathcal{M}}_{\mathbb{P}}^{1}(\mathcal E,T)$ such that ${\mathcal{F}}_{*}(\mu)\neq
 -\infty$ and $h_{\mu}^{(r)} (T) + {\mathcal{F}}_{*}(\mu) \geq
 \pi_{T}({\mathcal{F}})(\epsilon)$.

 Let $\epsilon >0$ be an arbitrary small number such that $\pi_{T}({\mathcal{F}})(\epsilon)\neq
 -\infty$. For any $n\in {\mathbf{N}}$, due to lemma 3.1, we can
 take a measurable in $\omega$ family of maximal
 $(\omega,\epsilon,n)$ separated sets $G(\omega,\epsilon,n) \subset
 \mathcal E_\omega$ such that
 \begin{eqnarray}\label{inequality2} \sum\limits_{x\in G(\omega,\epsilon,n)} e^{f_n(\omega,x)}
 \geq \frac{1}{e} \pi_{T}({\mathcal{F}})(\omega,\epsilon,n).
 \end{eqnarray}
 Next, define probability measures $\upsilon^{(n)}$ on $\mathcal E$ via
 their measurable disintegrations
 \[ \upsilon_{\omega}^{(n)} = \frac {\sum_{x\in
 G(\omega,\epsilon,n)}e^{f_n(\omega,x)}\delta_x} {\sum_{y\in
 G(\omega,\epsilon,n)}e^{f_n(\omega,y)}}
 \]
where $\delta_{x}$ denotes the Dirac measure at $x$, so that
${\mathrm{d}}\upsilon^{(n)} (\omega,x)= {\mathrm{d}}
\upsilon_{\omega}^{(n)} (x) {\mathrm{d}} {\mathbb{P}}(\omega)$,
and set
\[ \mu^{(n)} = \frac {1}{n} \sum\limits_{i=0}^{n-1} \Theta^{i}
\upsilon^{(n)}.
\]
By the definition of $\pi_{T}({\mathcal{F}})(\epsilon)$ and lemma
3.2 (i)-(ii), we can choose a subsequence of positive integers $\{
n_j \}$ such that
\begin{eqnarray}\label{equality}
\lim\limits_{j\rightarrow \infty} \frac{1}{n_j} \int \log
\pi_{T}({\mathcal{F}})(\omega,\epsilon,n_j) {\mathrm{d}}
{\mathbb{P}}(\omega)= \pi_{T}({\mathcal{F}})(\epsilon) \qquad and\
\mu^{(n_j)}\Rightarrow \mu\ as\ j\rightarrow\infty
\end{eqnarray}
for some $\mu\in {\mathcal{M}}_{\mathbb{P}}^{1}(\mathcal E,T)$.

Now we choose a partition ${\mathcal{A}}=\{ A_1,\dots,A_k \}$ of
$X$ with ${\mathrm{diam}}({\mathcal{A}}):=\max \{
{\mathrm{diam}}(A_j):1\leq j \leq k \} \leq \epsilon$ and such
that $\int \mu_{\omega}(\partial A_i)
{\mathrm{d}}{\mathbb{P}}(\omega) =0$ for all $1\leq i\leq k$,
where $\partial$ denotes the boundary. Set ${\mathcal{A}}(\omega)
=\{ A_1(\omega),\dots, A_k(\omega)\}, A_i(\omega) =A_i\cap
\mathcal E_\omega, 1\leq i\leq k$. Since each element of
$\bigvee\limits_{i=0}^{n-1} (T_{\omega}^{i})^{-1}
{\mathcal{A}}(\vartheta^{i}\omega)$ contains at most one element
of $G(\omega,\epsilon,n)$, we have by (\ref{inequality2}),
\begin{eqnarray*}
&&H_{\upsilon_{\omega}^{(n)}} (\bigvee\limits_{i=0}^{n-1}
(T_{\omega}^{i})^{-1} {\mathcal{A}}(\vartheta^{i}\omega)) + \int
f_n(\omega,x) {\mathrm{d}} \upsilon_{\omega}^{(n)}(x)\\
&&=\sum\limits_{y\in G(\omega,\epsilon,n)}
\upsilon_{\omega}^{(n)}(\{y\})(f_n(\omega,y)-\log
\upsilon_{\omega}^{(n)}(\{ y\}))\\
&&=\log \sum\limits_{y\in G(\omega,\epsilon,n)} e^{f_n(\omega,y)}\\
&&\geq \log \pi_{T}({\mathcal{F}})(\omega,\epsilon,n) -1.
\end{eqnarray*}
Let ${\mathcal{B}}=\{ B_1,\dots,B_k \}$, $B_i=(\Omega\times
A_i)\cap \mathcal E$. Then ${\mathcal{B}}$ is a partition of
$\mathcal E$ and $B_i(\omega)=\{ x\in \mathcal E_\omega :
(\omega,x)\in B_i \}= A_i(\omega)$. Integrating in the above
inequality against ${\mathbb{P}}$ and dividing by $n$, we have by
(2.2) the inequality
 \begin{eqnarray}\label{large}
 \frac{1}{n}H_{\upsilon^{(n)}} (\bigvee\limits_{i=0}^{n-1} (\Theta^{i})^{-1} {\mathcal{B}}\mid
 {\mathcal{W}}_{\mathcal E})+ \frac{1}{n}\int f_n {\mathrm{d}} \upsilon^{(n)} \geq
 \frac{1}{n}\int \log \pi_{T}({\mathcal{F}})(\omega,\epsilon,n)
 {\mathrm{d}}{\mathbb{P}}(\omega) -\frac{1}{n}.
 \end{eqnarray}
 Consider $q,n\in {\mathbf{N}}$ such that $1< q<n$ and for $0\leq l<q$ and  let
 $a(l)$ denote the integer part of $(n-l)q^{-1}$, so that
 $n=l+a(l)q+r$ with $0\leq r<q$. Then
 \[ \bigvee\limits_{i=0}^{n-1} (\Theta^{i})^{-1} {\mathcal{B}} =
 (\bigvee\limits_{j=0}^{a(l)-1}(\Theta^{l+jq})^{-1} \bigvee\limits_{i=0}^{q-1} (\Theta^{i})^{-1} {\mathcal{B}})
 \vee \bigvee\limits_{m\in S_{l}} (\Theta^{m})^{-1} {\mathcal{B}},
 \]
 where $S_{l}$ is a subset of $\{ 0,1,...,n-1\}$ with cardinality at most $2q$.
  Since card${\mathcal{B}}=k$, taking into account the
 subadditivity of conditional entropy(see \cite{ki}, section 2.1) it follows that
 \[ H_{\upsilon^{(n)}} (\bigvee\limits_{i=0}^{n-1} (\Theta^{i})^{-1} {\mathcal{B}}\mid
 {\mathcal{W}}_{\mathcal E})\leq \sum\limits_{j=0}^{a(l)-1}
 H_{\Theta^{l+jq}{\upsilon^{(n)}}} (\bigvee\limits_{i=0}^{q-1} (\Theta^{i})^{-1} {\mathcal{B}}\mid
 {\mathcal{W}}_{\mathcal E}) + 2q\log k.
 \]
 Summing here over $l\in \{ 0,1,\dots,q-1 \}$, we have
 \begin{eqnarray*} qH_{\upsilon^{(n)}} (\bigvee\limits_{i=0}^{n-1} (\Theta^{i})^{-1} {\mathcal{B}}\mid
 {\mathcal{W}}_{\mathcal E}) &\leq& \sum\limits_{m=0}^{n-1}
 H_{\Theta^{m}\upsilon^{(n)}} (\bigvee\limits_{i=0}^{q-1} (\Theta^{i})^{-1} {\mathcal{B}}\mid
 {\mathcal{W}}_{\mathcal E}) + 2q^{2} \log k\\
 &\leq& nH_{\mu^{(n)}}(\bigvee\limits_{i=0}^{q-1} (\Theta^{i})^{-1} {\mathcal{B}}\mid
 {\mathcal{W}}_{\mathcal E}) + 2q^{2} \log k
 \end{eqnarray*}
 where the second inequality relies on the general property of the
 conditional entropy of partitions $H_{\sum_{i}p_i \eta_i}(\xi\mid
 {\mathcal{R}})\geq$ $
 \sum_{i}p_i H_{\eta_i}(\xi\mid {\mathcal{R}})$ which holds for
 any finite partition $\xi$, $\sigma-$algebra ${\mathcal{R}}$,
 probability measures $\eta_i$, and probability vector
 $(p_i),i=1,\dots,n,$ in view of the convexity of $t\log t$ in the
 same way as in the unconditional case(see \cite{wal}, pp.
 183 and 188). Dividing by $nq$ in inequality as above, we have
 \[ \frac{1}{n}H_{\upsilon^{(n)}} (\bigvee\limits_{i=0}^{n-1} (\Theta^{i})^{-1} {\mathcal{B}}\mid
 {\mathcal{W}}_{\mathcal E})\leq \frac{1}{q}H_{\mu^{(n)}}(\bigvee\limits_{i=0}^{q-1} (\Theta^{i})^{-1} {\mathcal{B}}\mid
 {\mathcal{W}}_{\mathcal E}) + \frac{2q\log k}{n}.
 \]
 In particularly, we have
 \begin{eqnarray}\label{inequality3}
 \frac{1}{n_i}H_{\upsilon^{(n_i)}} (\bigvee\limits_{j=0}^{n_i-1} (\Theta^{j})^{-1} {\mathcal{B}}\mid
 {\mathcal{W}}_{\mathcal E})\leq \frac{1}{q}H_{\mu^{(n_i)}}(\bigvee\limits_{j=0}^{q-1} (\Theta^{j})^{-1} {\mathcal{B}}\mid
 {\mathcal{W}}_{\mathcal E}) + \frac{2q\log k}{n_i}.
 \end{eqnarray}

 Observe that the boundary of
 $\bigvee\limits_{i=0}^{q-1}(T_{\omega}^{i})^{-1}{\mathcal{A}}(\vartheta^{i}\omega)$
 is contained in the union of boundaries of
 $(T_{\omega}^{i})^{-1}{\mathcal{A}}(\vartheta^{i}\omega)$ and $\mu_{\omega}((T_{\omega}^{i})^{-1}\partial
 {\mathcal{A}}(\vartheta^{i}\omega))= \mu_{\vartheta^{i}\omega}(\partial {\mathcal{A}}(\vartheta^{i}\omega))
  {\mathbb{P}}-a.s.$. $\mu\in
 {\mathcal{M}}_{\mathbb{P}}^{1}(\mathcal E,T)$ implies that   $\mu_{\omega}(\partial
 \bigvee\limits_{i=0}^{q-1}(T_{\omega}^{i})^{-1}{\mathcal{A}}(\vartheta^{i}\omega))=0 \quad
 {\mathbb{P}}-a.s.$  Taking into account lemma 3.2(iii), we have
 \begin{eqnarray*} \limsup\limits_{i\rightarrow \infty} \frac{1}{q}
 H_{\mu^{(n_i)}}(\bigvee\limits_{j=0}^{q-1}(\Theta^{j})^{-1}{\mathcal{B}}\mid {\mathcal{W}}_{\mathcal E})
 \leq \frac{1}{q}
 H_{\mu}(\bigvee\limits_{j=0}^{q-1}(\Theta^{j})^{-1}{\mathcal{B}}\mid {\mathcal{W}}_{\mathcal
 E}).
 \end{eqnarray*}
Letting $i$ approach $\infty$ in (\ref{inequality3}), we have
\begin{eqnarray}\label{inequality4}
\limsup\limits_{i\rightarrow \infty}
\frac{1}{n_i}H_{\upsilon^{(n_i)}} (\bigvee\limits_{j=0}^{n_i-1}
(\Theta^{j})^{-1} {\mathcal{B}}\mid
 {\mathcal{W}}_{\mathcal E})\leq \frac{1}{q}H_{\mu}(\bigvee\limits_{j=0}^{q-1} (\Theta^{j})^{-1} {\mathcal{B}}\mid
 {\mathcal{W}}_{\mathcal E}).
 \end{eqnarray}
 From lemma 3.5, we know
 \begin{eqnarray}\label{inequality5}
 \limsup\limits_{i\rightarrow \infty} \frac{1}{n_i} \int f_{n_i}
 {\mathrm{d}} \upsilon^{(n_i)} \leq {\mathcal{F}}_{*}(\mu).
 \end{eqnarray}
Combining (\ref{equality}),(\ref{large}),(\ref{inequality4}) with
(\ref{inequality5}), we obtain
\[ \frac{1}{q}H_{\mu}(\bigvee\limits_{j=0}^{q-1} (\Theta^{j})^{-1} {\mathcal{B}}\mid
 {\mathcal{W}}_{\mathcal E}) +{\mathcal{F}}_{*}(\mu) \geq
 \pi_{T}({\mathcal{F}})(\epsilon).
 \]
 Letting $q\rightarrow \infty$, we have
 \[ \pi_{T}({\mathcal{F}})(\epsilon) \leq
 h_{\mu}^{(r)}(T,{\mathcal{B}}) + {\mathcal{F}}_{*}(\mu) \leq
 h_{\mu}^{(r)}(T) + {\mathcal{F}}_{*}(\mu).
 \]
This completes the proof of step 2.

Step 3: $\pi_{T}({\mathcal{F}})=-\infty$ if and only if
${\mathcal{F}}_{*}(\mu)=-\infty$ for all $\mu\in
{\mathcal{M}}_{\mathbb{P}}^{1}(\mathcal E,T)$.

By step 1 we have $\pi_{T}({\mathcal{F}}) \geq h_{\mu}^{(r)}(T) +
{\mathcal{F}}_{*}(\mu)$ for all $\mu\in
{\mathcal{M}}_{\mathbb{P}}^{1}(\mathcal E,T)$ with
${\mathcal{F}}_{*}(\mu)\neq -\infty$, which shows the necessity.
The sufficiency is implied by step 2(since if
$\pi_{T}({\mathcal{F}})\neq -\infty$, then by step 2 there exists
some $\mu$ with ${\mathcal{F}}_{*}(\mu)\neq -\infty$). This
completes   the proof of the theorem.
\end{proof}

\section{The Hausdorff dimension for asymptotic conformal repellers}

In this section, we consider the Hausdorff dimension for repeller in  random dynamical system(RDS). Precisely,
fix an ergodic invertible transformation $\vartheta$ of a probability space $(\Omega,
\mathcal{W},\mathbb{P})$ and let $M$ be  a compact Riemann manifold. We consider a measurable
 family $T = \{ T_{\omega}: M \to M \} $ of $C^1$ maps, i.e.
$(\omega, x) \mapsto T_{\omega} x $ is assumed to be measurable. This determines a differentiable RDS via
$T_{\omega}^{n}={T_{\vartheta^{n-1}\omega } \cdots
T_{\vartheta\omega} T_{\omega}},\ \ n\geq 1$.
 $ \mathcal E \subset
\Omega\times M$ is a measurable set and such that  all the fibers
( sometimes called $\omega$-sections ) $\mathcal E_{\omega}=\{
x\in X \mid (\omega,x)\in \mathcal E\}$ are compact. $\mathcal E$
is said to be invariant with $T$ if
 $T_{\omega} \mathcal E_{\omega} = \mathcal E_{\vartheta\omega} \,\,\ \  \mathbb{P}-a.s.$  In
 \cite{kifer-tran}, \cite{bog2}, the authors consider the Hausdorff dimension for repeller
 in $C^{1+\alpha}$ conformal random dynamical system. They prove that, if $T_{\omega}$ is $C^{1+\alpha}$ conformal for
 $ \mathbb{P}-a.s$  and  $ \mathcal E \subset \Omega\times M$ is a repeller which is  invariant with $T$  for
 random dynamical system,
 then the Hausdorff dimension
can be obtained as the zero $t_0$ of $t \mapsto \pi_T (-t \log\|D_x T \|)$, where $\pi_T (-t \log\|D_x T \|)$ is
topological pressure for random dynamical system $T$ with additive potential $-t \log\|D_x T \|$.

A  repeller is called conformal if  $T_{\omega}$ for
 $ \mathbb{P}-a.s$ is conformal. In some sense, conformality   in  random dynamical systems is strong.
  Now we give a  definition of  asymptotically  conformal repeller, which is weaker than conformal repeller.
 Let $\mathcal M_{\mathbb{P}}(\mathcal E)$ be the set of $T$ invariant probability measures on $\mathcal E$
 whose marginal on $\Omega$ coincides with $\mathbb{P}$ and $E_{\mathbb{P}}(\mathcal E)$ be the set of $T$ invariant
 ergodic  probability measures on $\mathcal E$
 whose marginal on $\Omega$ coincides with $\mathbb{P}$. By the Oseledec multiplicative ergodic theorem \cite{ose}, for any
 $\mu \in E_{\mathbb{P}}(\mathcal E)$, we can define Lyapunov exponents $\lambda_1(\mu) \le \lambda_2(\mu) \le
  \cdots \le \lambda_d(\mu), \, \, \,  d=dim M$.  An invariant  repeller  for
 random dynamical system is called asymptotically conformal if for any $\mu \in E_{\mathbb{P}}(\mathcal E)$, $\lambda_1(\mu)=
 \lambda_2(\mu) = \cdots = \lambda_d(\mu)$. It is obvious that a conformal repeller is an asymptotically conformal repeller, but
  reverse isn't true. Using topological pressure of random bundle transformations in sub-additive case, we can obtain
  the Hausdorff dimension for asymptotically conformal repeller. We state the result as follows, and the proof will be
  given in
  the forthcoming paper.

\begin{theorem} Let $T$ be $C^{1+\alpha}$ random dynamical system and $\mathcal E$ be an asymptotically conformal repeller.
Then the Hausdorff dimension of $\mathcal E$ is zero $t^*$ of $t
\mapsto \pi_T (-t \mathcal F)$, where $\mathcal F = \{ \log m(D_x
T_{\omega}^n), \ \ (\omega, x)  \in \mathcal E, n \in \mathbb{N}
\}$ and $m(A)=\|A^{-1}\|^{-1}$.
\end{theorem}

\begin{remark}
 If $\mathcal E$ isn't asymptotically conformal repeller, we can obtain the upper estimate of the Hausdorff
dimension by using topological pressure of random bundle transformations in sub-additive case.
\end{remark}


\begin{thebibliography}{999}

\bibitem{barreira1} L.Barreira,  A non-additive thermodynamic formalism and
applications to dimension theory of hyperbolic dynamical systems
Ergodic Theory Dyn. Syst. 16(1996) 871--927.

\bibitem{barreira}L.Barreira,  Dimension estimates in nonconformal hyperbolic
dynamics.   Nonlinearity 16 (2003), no. 5, 1657--1672.

\bibitem{bog}
T.Bogensch\"{u}tz,   Entropy, Pressure, and a Variational
Principle for Random Dynamical Systems,  Random Comput. Dynam.,
1(1)(1992-93),99-116 .

\bibitem{bog2}
T.Bogensch\"{u}tz, G.Ochs, The Hausdorff dimension of conformal
repellers under random perturbation, Nonlinearity 12(1999),
1323-1338.

\bibitem{bowen}
R. Bowen, Hausdorff dimension of quasicircles, Inst. Hautes
\'{E}tudes Sci. Publ. Math.,50(1979), 11--25.

\bibitem{caofh}
Yongluo Cao, Dejun Feng, Wen Huang,  The Thermodynamic Formalism
for Sub-multiplicative Potentials, To appear in Discrete and 
Continuous Dynamical Systems (A)(2007).

\bibitem{cao} Yongluo Cao, Dimension upper bounds estimate in non-conformal hyperbolic invariant
set, Preprint (2007).

\bibitem{falconer} K.Falconer. The Hausdorff dimension of self-affine
fractals.  Math. Proc. Camb. Phil. Soc. 103(1988) 339--50.

\bibitem{Fal88} K.J. Falconer, A subadditive thermodynamic formalism for
mixing repellers. {\it J. Phys. A} {\bf 21} (1988), no. 14,
L737--L742.

\bibitem{ki}
Y.Kifer,  Ergodic theory of random transformations, Birkhauser,
Boston, 1986.

\bibitem{kifer}
Y.Kifer,  On the topological pressure for random bundle
transformations, In Rohlin's memorial Volume ( V.Turaev and
A.Vershik eds.) Amer. Math. Soc. Transl.(2)Vol.202(2001), 197-214.

\bibitem{kifer-tran}
Y.Kifer, Fractal dimensions and random transformations,
Transactions of the American Mathematical Society, 348(1996),
2003-2038.

\bibitem{led-walters} F.Ledrappier and P.Walters, A relativized
variational principle for continuous transformations, J.London
Math.Soc. (2) 16(1997), 568-576.

\bibitem{ose} V.I. Oseledec, A multiplicative ergodic theorem, Lyapunov characteristic number for dynamical systems,
Trans.Moscow Math.Soc. 19(1968),197-221.

\bibitem{ruelle}
D. Ruelle, Statistical mechanics on a compact set with $Z\sp{v}$
action satisfying expansiveness and specification,  Trans. Amer.
Math. Soc., 187(1973), 237--251.

\bibitem{walter}P. Walters,
A variational principle for the pressure of continuous
transformations, Amer. J. Math., 97(1975), 937--971.

\bibitem{wal}
P.Walters, An Introduction to Ergodic Theory, Springer-Verlag, New
York,1982.

\bibitem{zhang} Yingjie Zhang,  Dynamical upper bounds for Hausdorff
dimension of invariant sets, Ergodic Theory Dynam. Systems 17
(1997), no. 3, 739--756.


\end{thebibliography}
  \end{document}